\documentclass[a4paper]{article}
\usepackage {amssymb,latexsym}
\usepackage {amsmath}
\newtheorem{thm}{Theorem}
\newtheorem{prop}{Proposition}
\newtheorem{cor}{Corollary}

\newtheorem{de}{Definition}

\newenvironment{proof}{
                        \noindent{\bf\small Proof: }\small}
                                       {\hfill {$\mathbf \Box$}\medskip}

\newcommand{\K}{\mathbb{K}}
\newcommand{\N}{\mathbb{N}}

\newcommand{\R}{\mathbb{R}}

\newcommand{\I}{\mathrm{I}}

\newcommand{\n}{\mathfrak{n}}

\newcommand{\g}{\mathfrak{g}}

\makeatletter
\@addtoreset{equation}{section}

\makeatother
\title{Index of a Borel Lie Algebra
}
\author{
\\Toukaiddine Petit\footnote{{\tt  toukaiddine.petit@icloud.com}. 
}}

\date{09/01/2024}

\begin{document}

\maketitle
\begin{abstract}
We compute the index of a Lie Borel Lie Algbra of a simple Lie algebra.

\end{abstract}

\section{Introduction}
The classification of Lie algebras for fixed dimensional is a difficul problem. This problem implies that we compare  Lie algebras. The using the invariant of Lie algebras help to compare two Lie algebras. The index of Lie algebra is an invariant. The main this paper is to compute the index of a Borel Lie algebra of a Simple Lie Algebra over the filed $\R$. In the case the simple Lie Algebras are $B_r,C_r,D_{2r},E_8,E_7,F_4,G_2$, the index of a Borel Lie algebra is $0$. In this case we say that the Lie algebra is Frobenuis.
\section{Maximal set of Srongly Orthogonal Roots}
 Let $\g$ be a semisimple Lie Algebra over $\R$. We fix $\bold{h}$ a Cartan subalgebra of $\g$. Let $\bold{h}^*$ the dual space of $\bold{h}$. 
Let $\Delta$ (resp.; $\Delta_+, \Delta_{-}$) be the set of roots (resp. positive, negative). Let $ W(\Delta)$ be the Weyl group. Let 
 $\pi=\lbrace\alpha_1,\alpha_2,...,\alpha_r\rbrace$ be the simple of roots associated with $(\Delta,\bold{h})$. Let $H_1,..,H_r$ be the simple coroots.
Let $\alpha\in \Delta$ a root and
 $$\g^{\alpha}=\lbrace X\in\g:[X,H]=\alpha(H)X,\forall \in\bold{h} \rbrace$$
  the eigenspace of $\alpha$.  The dimensional of the vector space $\g^{\alpha}$ is one. \\
 
We put $$\n_+=\oplus_{\alpha\in\Delta_+}\g^{\alpha},\n_{-}=\oplus_{\alpha\in\Delta_{-}}\g^{\alpha},
\bold{b}=\bold{h}\oplus\n_+$$
The subalgebra $\bold{b}$ is called a Borel Lie algebra. The Lie algebra $\g$ can be decomposed as
$\bold{g}=\bold{h}\oplus\n_+\oplus\n_{-}$
For each $i\in\lbrace1,\dots,r\rbrace$, we choose vectors $X_i\in\g^{\alpha_i}$ and $Y_i\in\g^{-\alpha_i}$ such that
$[X_i,Y_i]=H_i $. We put $n(i,j)=\alpha_i(H_j)$. We have $n(i,j)=0$ if $i\neq j$ and $n(i,i)=2$. The matrix formed by the number $(n(i,j))$ is called Cartan matrix.
The Weyl basis of $\g$ formed by the vectors $(X_i,Y_i,H_i)$ verified the following properties
$$[H_i,H_j]=0,[X_i,Y_i]=H_i,[H_i,X_j]=n(i,j)X_j,[H_i,Y_j]=-n(i,j)Y_j,$$
$$ [X_i,Y_j]=0 , i\neq j,$$
where $(X_i)$ is a basis of $\n_+$, $(Y_j)$ is a basis of $\n_{-}$.
\\ We have also if $\alpha,\beta\in\Delta$ such that $\alpha+\beta\in\Delta$, then $[\alpha,\beta]=N_{\alpha,\beta}X_{\alpha+\beta}$
\\ where $N_{\alpha,\beta}=N_{-\alpha,-\beta}$ and $N_{\alpha,\beta}=q(1-p)\alpha(H_\alpha)/2$ where $p$ and $q$ are difined as $\beta+k\alpha$ is root if $k\in[p,q]$ and $\beta+k\alpha$ is not root if $k\notin[p,q]$.
\\
We define a partiel relation ordre on $\Delta$ as $\alpha>\gamma$ if $\alpha-\gamma=\sum n_i\alpha_i\in\Delta_+$\\
Let $(,)$ be an invariant nondegerete symmetric form of $\g^*$. 
\begin{prop}If $\Delta$ is irrecducible (i.e $\g$ is simple Lie Algebra), then
\begin{enumerate}
\item
there is the highest root $\beta$ of $\Delta$ satisfying $\beta=\sum n_i\alpha_i\in\Delta_{+}$ with $n_i\geqslant 1$; and $(\beta,\alpha_i)\geqslant 0$, for all $i$.
\item
we have $(\beta,\beta)\geqslant (\alpha,\alpha)$ for any $\alpha\in\Delta$; 
\item 
we have $n(\beta,\alpha)=2(\alpha,\beta)/(\alpha,\alpha)=-1,0,1$ for any $\alpha\in\Delta$  $\alpha\neq\pm\beta$;
\item  if $(\alpha,\beta)>0$ then $\beta-\alpha$ is a root, for any $\alpha\in\Delta_{+}$, $\alpha\neq\beta$;
\item if $(\alpha,\beta)<0$ then $\beta+\alpha$ is a root, for any $\alpha\in\Delta_{-}$, $\alpha\neq-\beta$.
\end{enumerate}
\end{prop} 
\begin{de}

One says that two roots $\alpha$ and $\gamma$ are strongly orthogonal if $\alpha+\gamma\in\Delta$ and $\alpha-\gamma\in\Delta$\\
A subset $\Gamma$ of $\Delta$ is said strongly orthogonal if for all $\alpha$ and $\gamma$ are strongly orthogonal. 
\end{de}
\begin{prop}
Let  $\alpha$ and $\gamma$ be two roots strongly orthogonal of $\Delta$. Then we have $(\alpha,\gamma)=0$.
\end{prop}

\begin{prop} We suppose that $\Delta$ is irrecducible. Let $\beta$ be
 the highest root of $\Delta$. We set $\beta^{\perp}:=\lbrace\alpha\in\Delta:(\alpha,\beta)=0\rbrace$.
 \begin{enumerate}
 \item For all $\alpha\in\beta^{\perp}$, then $\alpha$ and $\beta$ are strongly orthogonal.
 \item $\beta^{\perp}$ is a root system of a Lie semisimple sub
 algebra $\g_{\beta}$ of $\g$.
 \end{enumerate}
\end{prop}
\begin{thm}
\begin{enumerate}
\item Let $(\Delta_{i})_{i}$ be the sum of irrecduble system of $\Delta$. For each $i$, let $\beta_i$ be the highest root of $\Delta_i$, and let $(\Delta_{ij})_{j}$ be the sum of irrecduble system of $\beta^{\perp}_{i}$, the set of orthogonal roots to $\beta_{i}$ in $\Delta_i$. We fixe $i$ and $j$, let $\beta_{ij}$ be the highest root of $\Delta_{ij}$, and let $(\Delta_{ijk})_{k}$ be the sum of irrecduble system of $\beta^{\perp}_{ij}$, the set of orthogonal roots to $\beta_{ij}$ in $\Delta_{ij}$. Then the process defined a subset $I(\Delta)=\lbrace i,ij,ijk,...\rbrace$ and the set $(\beta_s)_{s\in I(\Delta)}$ is a maximum set of strongly roots of $\Delta$.
\item Let $\g$ be a semisimple Lie algebra. Two maximum sets of strongly roots of $\Delta$ are isomorphic by the Weyl group $W(\Delta)$.
\end{enumerate}
\end{thm}

\begin{prop}

\begin{enumerate}
\item We define a partiel ordre on $I(\Delta)$ by: \\
$l\leq k$ if there is $l_1,..,l_t\in\N$ such that $k=l,l_1,...,l_t$, with $l,k\in I(\Delta)$.
\item For all $K\in I(\Delta))$, there is a unique maximal element $L$ of $I(\Delta)$ such that $L<K$.
\end{enumerate}

\end{prop}
\begin{prop}
A subset $\Gamma$ of $\Delta$ is called parabolic if for all $L\in \Gamma$, we have $K\in\Gamma$ for all $K\leq L$. Let $K\in I(\Delta)$, then $I_K=\lbrace L\in I(\Delta):L\leq K \rbrace$ is parabolic.
\end{prop}
Let $K\in I(\Delta)$, we set:
$$ \Delta^{+}_K:= \Delta^{+}\cap\Delta_K$$
$$\Gamma_K:=\lbrace\gamma\in\Delta_K:(\gamma,\beta_K)>0\rbrace$$

\begin{prop}Let $L,K\in\I(\Delta)$. We have
\begin{enumerate}
\item $\Gamma_K$ is the complement of $\beta_K^{\perp}$ in $\Delta_K$,and $\Delta^{+}$ is disjoint union of $\Gamma_K$, $K\in I(\Delta)$.
\item Let $\gamma\in\Gamma_K$ and $\delta\in\Gamma_L$ such that $\gamma+\delta\in\Delta$, then $\gamma+\delta\in\Gamma_K$ ($K\leq L$) or $\gamma+\delta\in\Gamma_L$ ($L\leq K$).
\item For all $\gamma\in\Gamma_K$, $\gamma\neq\beta_K$, and 
$(\beta_K,\gamma)=\frac{1}{2}(\beta_K,\beta_K)$, $\beta_K-\gamma\in\Gamma_K$, $\beta_K-\gamma\neq\beta_K$
\item Let $\gamma,\delta\in\Gamma_K-\lbrace\beta_K\rbrace$ such that $\gamma+\delta\in\Delta$, then $\gamma+\delta=\beta_K$.
ty of the highest root $\beta_K$, we have $\beta_K=\gamma+\delta$. 
\end{enumerate} 
\end{prop}
\begin{proof}
\begin{enumerate}
\item[1]Let $K\in I(\Delta)$. Let $\alpha\in\Gamma_K$ then  $\alpha\in\Delta^{+}_K$ or $\alpha\in\Delta^{-}_K$. Since $\beta_K$ is the highest root of $\Delta_K$, then $\alpha\in\Delta^{+}_K$ because if $\alpha\in\Delta^{-}_K$, then we have $(\alpha,\beta_K)<0$. Then $\Gamma_K\subset\Delta^{+}_{K}$ and $\Delta^{+}_K=\Gamma_K\cup\beta_K^{\perp}$ and $\Gamma_K\cap\beta_K^{\perp}=\emptyset$. By construction of $I(\Delta)$, we have $\Delta^{+}=\cup_{K\in I(\Delta)}\Gamma_K$, $\Gamma_K\cap\Gamma_L=\emptyset$, for $K\neq L$.
\item[2] Let $\gamma\in\Gamma_K$ and $\delta\in\Gamma_L$ such that $\gamma+\delta\in\Delta$. We assume that $\gamma+\delta\Gamma_K$ nor $\gamma+\delta\Gamma_L$. Then $\Gamma_K$ and $\Gamma_L$ come from  two components of  a $\Delta_N$ where $N\in I(\Delta)$. Consequently $\gamma+\delta$ dot not below to $\Delta$.

 \item [3] 
\item [4] Let $\gamma,\delta\in\Gamma_K-\lbrace\beta_K\rbrace$ such that $\gamma+\delta\in\Delta$, then $(\beta_K,\gamma)=\frac{1}{2}(\beta_K,\beta_K)$ and $(\beta_K,\delta)=\frac{1}{2}(\beta_K,\beta_K)$. Then $(\beta_K,\delta+\gamma)=(\beta_K,\beta_K)$. By unicity of the highest root $\beta_K$, we have $\beta_K=\gamma+\delta$.
\end{enumerate}
\end{proof}

\begin{prop}
Let $P$ be a system of roots of the form $\beta^{\perp}$.
\begin{enumerate}
\item If $\gamma\in Q^{+}$, $\delta\in\Delta$ such that $\gamma+\delta\in\Delta^{+}$, then $\gamma+\delta\in Q^{+}$.
\item If $Q^{+}\cap\Gamma_K\neq \emptyset$, then $\beta_K\in Q^{+}$.
\item $I(Q^{+}):=\lbrace L\in I(\Delta):\beta_K\in Q^{+}\rbrace$ is a subset parabolic of $\Delta$.
\item Let $\gamma\in\Gamma_K(P)$, $\alpha\in Q^{+}$, then $\alpha+\gamma\in\Delta_{+}$, then $\alpha\in\Gamma_L$, with $L<K$. 
\end{enumerate}
\end{prop}
\begin{thm} Let $\g$ a simple Lie algebra. The maximum set of strongly roots of $\Delta$ is given by 
\begin{enumerate}
 \item $\g=A_r:$  $\sum_{k=j}^{r+1-j}\alpha_k$, $j=1...[\frac{r+1}{2}]$, $car(I)=[\frac{r+1}{2}]$\\
 \item$\g=C_r:$  $2\sum_{k=j}^{r-1}\alpha_k+\alpha_r$, $j=1...r$, $car(I)=r$
 \item $\g=E_6:$ $\alpha_1+2\alpha_2+2\alpha_3+3\alpha_4+2\alpha_5+\alpha_6$, $\alpha_1+\alpha_3+\alpha_4+\alpha_5+\alpha_6$,$\alpha_3+\alpha_4+\alpha_5$,$\beta_4=\alpha_4, card(I)=4$
 \item$\g=F_4:$ $2\alpha_1+3\alpha_2+4\alpha_3+2\alpha_4,\alpha_2+2\alpha_3+2\alpha_4,\alpha_2+2\alpha_3,\alpha_2,car(I)=4$
 \item $\g=G_2:$ $3\alpha_1+2\alpha_2$, $\alpha_1,car(I)=2$
 \item
 \begin{enumerate}
 \item $\g=B_{2r}:$, $\alpha_{2j-1}+2\sum_{j=2j}^{2r}\alpha_{i}$,  $\alpha_{2j-1}$, $j=1..r$, $car(I)=2r$.
 \item
  $\g=B_{2r+1}:$ $\alpha_{2j-1}+2\sum_{j=2j}^{2r}\alpha_{i}, j=1...r$,$\alpha_{2j-1}$, $j=1..r$,$\alpha_{2r+1}$, $car(I)=2r+1$
  \end{enumerate}
  \begin{enumerate}
  \item  $\g=D_{2r}:$ $\alpha_{2j-1}+2\sum_{2j}^{2r-2}\alpha_{i}+\alpha_{2r-1}+\alpha_{2r}$
$j=1...r-1$, $\beta_j=\alpha_{2j}$
\item $\g=D_{2r+1}:$ $\alpha_{2j-1}+2\sum_{2j}^{2r-2}\alpha_{i}+\alpha_{2r-1}+\alpha_{2r+1}$
$j=1...r-1$, 
\end{enumerate}
\item $\g=E_7:$ $2\alpha_1+2\alpha_2+3\alpha_3+4\alpha_4+3\alpha_5+2\alpha_6+\alpha_7$, $\alpha_2+\alpha_3+2\alpha_4+2\alpha_5+2\alpha_6+\alpha_7$, $\alpha_2+\alpha_3+2\alpha_4+\alpha_5$,$\alpha_3$,$\alpha_7$,$\alpha_5$,$\alpha_2$, $card(I)=7$.
\item $\g=E_8:$ $2\alpha_1+3\alpha_2+4\alpha_3+6\alpha_4+5\alpha_5+4\alpha_6+3\alpha_7+2\alpha_8$, $2\alpha_1+2\alpha_2+3\alpha_4+2\alpha_5+2\alpha_6+\alpha_7$, $ \alpha_2+\alpha_3+2\alpha_4+2\alpha_5+2\alpha_6+\alpha_7$, $\alpha_2+\alpha_3+2\alpha_4+\alpha_5$, $\alpha_3$,$\alpha_7$,$\alpha_5$,$\alpha_2$, $car(I)=8$.
\end{enumerate}
\end{thm}

\section{Index of Borel Lie Algebras}
Let $\g$ a Lie algebra over $\K$ and $\g^*$ its dual. Let $f$ be a element of $\g^*$, a form linear on $\g$. We note by $B_f$ the antisymetric bilinear form on $\g$ defined by:
\\ $$(X,Y)\mapsto f([X,Y])$$
The dual space $\g^*$ has a natural representation called coadjoint represenation
$$ ad^*_X:f\mapsto-B_f(X,.)$$
\begin{de}
The index of the form $B_f$ is defined by the formula:
$$\bold{i}(B_f)=\dim(KerB_f)$$

\end{de}
\begin{de} 

The index of a Lie algebra $\g$ is defined by
-$$ \bold{i}(\g)=\min_{f\in\g^*}(\bold{B_f})$$
One said that the form $f$ is regular if $\bold{i}(B_f))=\bold{i}(\g)$. We denote $\g^*_{reg}$ the set of regular forms on $\g$.
\end{de}
 \begin{prop}The set of regular forms $\g_{reg}^*$ is Zariski open.
 \end{prop}
 \begin{de}A Lie algebra $\g$ is called Frobenuis if $\bold{i}(\g)=0$
 \end{de}
 \begin{prop}Let $\g$ a Lie algebra. One assume that $\g$ be the product of $\g_1, g_2,..,\g_r$ of Lie algebras.
 \\
 A form $f=(f_1,..,f_r)$ is a regular form of $\g^*$ if only if $f_i$ is regular of $\g_i^*$.
 \end{prop}
 
 Let $\bold{g}$ be a Lie semisimple Lie Algebra. Let $\bold{b}$ a Borel Lie subalgebra of $\g$ and $\bold{n}$ its nilradical. Let $\bold{p}$ a parabolic subalgebra of $\g$ containing $\bold{b}$ and $\bold{m}$ its nilradical. Let $Q^{+}$ a subset of $\Delta_{+}$ such that $\bold{m}=\g^{Q{+}}=\lbrace X_{\alpha}:\alpha\in Q^{+}\rbrace$. Let $F=(\beta_k)_{k\in{I(\Delta)}}$ a maximum set of stronly orthogonal roots. Let $H_L$ the coroot of $\beta_L$, $L\in\Delta$.\\
 We set
$$I(Q^{+})=\lbrace L\in I(\Delta_{+}):\beta_L\in Q^{+}\rbrace ,\bold{h}_{I(Q^{+})}=\lbrace H_{L}:L\in I(Q^{+})\rbrace$$ $$\Gamma_L=\lbrace \gamma\in\Delta:(\gamma,\beta_L)> 0\rbrace,\bold{m}_{I(Q^{+})}=\bigoplus_{L\in I(Q^{+})}\g^{\Gamma_L}$$ $$\bold{d}_{\bold{m}}=\bold{h}_{I(Q^{+})}\oplus\bold{m}_{I(Q^{+})}$$ \\
It is easy to verify that $\bold{d}_{\bold{m}}$ is a subalgebra of $\g$ which is in $\bold{b}$, and $[\bold{d}_{\bold{m}},\bold{m}]\subset \bold{m}.$
\begin{thm}Let $\g$ be a semisimple Lie algebra $\g$, $\bold{b}$ a Borel Lie subalgebra of $\g$,  $\bold{p}$ be a parabolic subalgebra of $\g$ containing the Borel subalgebra $\bold{b}$, and $\bold{m}$ its nilradical. Then
\item[1.] The subalgebra $\bold{d}_{\bold{m}}$ is an ideal of $\bold{b}$ and its nilradical $\bold{m}_{I(Q^{+})}$ containing $\bold{m}$.
\item [2.] The subalgbra $\bold{d}_{\bold{m}}$ is Frobenuis Lie algebra.
\end{thm}Let $\g$ be a semisimple Lie algebra $\g$, $\bold{b}$ a Borel Lie subalgebra of $\g$,  $\bold{p}$ be a parabolic subalgebra of $\g$ containing the Borel subalgebra $\bold{b}$, and $\bold{m}$ its nilradical. Let $Aut(\g)$ be the group of automorphisms of $\g$. The subgroup of  $Aut(\g)$ genereted by the automorphisms $e^{adX}$, where $X$ is an element of $\g$ such that $ad X$ is nilpotent. This subgroup is denoted by $\bold{G}$. 
\begin{thm}Let $\g$ be a simple Lie algebra. Let $\bold{b}$ a Borel Lie algebra of $\g$. Let $(\beta_k)_{k\in I(\g)}$ be the maximum set of stronlgy roots. Let $i(\bold{b})$ be the index of $\bold{b}$. Then
 $f=\sum_{k\in I(\g)}X^{*}_{\beta_{k}}$ is a regular form of $\bold{b}^{*}_{reg}$. 
\end{thm}

\begin{thm}Let $\g$ be a simple Lie algebra. Let $\bold{b}$ a Borel Lie algebra of $\g$. Let $i(\bold{b})$ be the index of $\bold{b}$. Then
\begin{enumerate}
\item if $\g$ is type $A_r$, $i(\bold{b})=[\frac{r}{2}]$;
\item if $\g$ is type $B_r,C_r,D_{2r},E_8,E_7,F_4,G_2$,$i(\bold{b})=0$;
\item if $\g$ is type $D_{2r+1}$, $i(\bold{b})=1$;
\item if $\g$ is type $E_6$, $i(\bold{b})=2$.
\end{enumerate}
\end{thm}

\begin{cor}Let $\g$ be a simple Lie algebra. Let $\bold{b}$ a Borel Lie algebra of $\g$. If $\g$ is type $B_r,C_r,D_{2r},E_8,E_7,F_4,G_2$, then $\bold{b}$ is Frobenuis Lie algebra.
\end{cor}
\section{Theorem of Richardson}
Let $\bold{g}$ be a Lie semisimple Lie Algebra. Let $\bold{b}$ a Borel Lie subalgebra of $\g$ and $\bold{n}$ its nilradical. Let $\bold{p}$ a parabolic subalgebra of $\g$ containing $\bold{b}$ and $\bold{m}$ its nilradical. Let $Q^{+}$ a subset of $\Delta_{+}$ such that $\bold{m}=\g^{Q{+}}=\lbrace X_{\alpha}:\alpha\in Q^{+}\rbrace$. Let $F=(\beta_k)_{k\in{I(\Delta)}}$ a maximum set of stronly orthogonal roots. Let $H_L$ the coroot of $\beta_L$, $L\in\Delta$.\\
We set
$$I(Q^{+})=\lbrace L\in I(\Delta_{+}):\beta_L\in Q^{+}\rbrace ,\bold{h}_{I(Q^{+})}=\lbrace H_{L}:L\in I(Q^{+})\rbrace$$ $$\Gamma_L=\lbrace \gamma\in\Delta:(\gamma,\beta_L)> 0\rbrace,\bold{m}_{I(Q^{+})}=\bigoplus_{L\in I(Q^{+})}\g^{\Gamma_L}$$ $$\bold{d}_{\bold{m}}=\bold{h}_{I(Q^{+})}\oplus\bold{m}_{I(Q^{+})}$$ \\

Let $Aut(\g)$ be the group of automorphisms of $\g$. The subgroup of  $Aut(\g)$ genereted by the automorphisms $e^{adX}$, where $X$ is an element of $\g$ such that $ad X$ is nilpotent. This subgroup is denoted by $\bold{G}$. 
\begin{thm}Let $\g$ be a semisimple Lie algebra $\g$, $\bold{G}$ the group of its elementary hommorphismes, $\bold{b}$ a Borel Lie subalgebra of $\g$,  $\bold{p}$ be a parabolic subalgebra of $\g$ containing the Borel subalgebra $\bold{b}$, and $\bold{m}$ its nilradical. We set
$$\bold{P}=\lbrace g\in \bold{G}:g(\bold{p})=\bold{p}\rbrace$$
$$\bold{B}=\lbrace g\in \bold{G}:g(\bold{b}=\bold{b}\rbrace$$
\begin{enumerate}
\item  there is a unique nilpotent $\bold{G}$-orbit $\bold{O}$ in $\g^{*}$ such that 
\begin{enumerate}
\item $\bold {O}\cap\bold{m}^{*}$ is a dense open of $\bold{m}^{*}$ and $\bold {O}\cap\bold{m}^{*}=(\bold{d}_{\bold{m}})_{reg}\cap\bold{O}$
\item $\bold {O}\cap\bold{d}_{\bold{m}^{*}}$ is a dense open of $\bold{d}_{\bold{m}^{*}}$ and $\bold{O}\cap\bold{d}_{\bold{m}^{*}}=(\bold{d}_{\bold{m}})_{reg}$
\end{enumerate}
\item the set $\bold {O}\cap\bold{m}^{*}$ is a $\bold{P}$-orbit in $\bold{m}^{*}$ and the set $\bold {O}\cap\bold{m}^{*}$ is $\bold{B}$-orbit in $\bold{d}_{\bold{m}^{*}}$.
\item If $f\in \bold{O}\cap\bold{d}_{\bold{m}^{*}}$, then $f\circ ad \bold{b}=\bold{d}_{\bold{m}^{*}}$.
\end{enumerate}
\end{thm}


\end{document}